\documentclass[11pt]{article}
\usepackage[a4paper,margin=1in]{geometry}
\usepackage{amsmath,amssymb,amsthm,mathtools}
\usepackage{microtype}
\usepackage[colorlinks=true,linkcolor=blue,citecolor=blue,urlcolor=blue]{hyperref}

\newtheorem{theorem}{Theorem}
\newtheorem{corollary}[theorem]{Corollary}
\newtheorem{proposition}[theorem]{Proposition}
\newtheorem{lemma}[theorem]{Lemma}

\theoremstyle{remark}
\newtheorem{remark}[theorem]{Remark}

\newcommand{\D}{\mathcal D}
\newcommand{\Res}{\operatorname{Res}}
\newcommand{\C}{\mathbb C}
\newcommand{\Q}{\mathbb Q}
\newcommand{\norm}[1]{\left\lVert #1\right\rVert}

\title{Rational Bishop determinants and explicit cyclicity criteria:\\
Fourier parity, an infinite zero family, and quantitative gap functions}
\author{MA, YICEN}
\date{August 28, 2026}

\begin{document}
\maketitle

\begin{abstract}
We study finite-fibre determinants for rational Bishop operators and their
role in cyclicity for irrational parameters.  The paper has two main parts.
First, for the constant vector $f=1$, a resultant identity and a discrete
Fourier factorization reveal a determinant parity mechanism for general
modular orbit order: odd denominators give a nonnegative normalized
determinant on the positive fundamental cell, while for even denominators the
unique real alternating Fourier mode is the only factor capable of producing
a sign-changing zero.  We give an exact instance at $(r,q)=(9,16)$ and an
analytic infinite family $(r,q)=(3,6n-2)$.  Grivaux's zero-free determinant is
identified as the consecutive-order subfamily $\D_{1,q}$, so these zeros are
caused specifically by nonconsecutive modular ordering.

Second, we prove an explicit cyclicity criterion that does not assume global
nondegeneracy or monotonicity of the fibre determinant.  A quantitative Remez
estimate controls the small-determinant set; cutoff inverses are approximated
by endpoint-corrected Fej\'er polynomials; and an explicit continuity modulus
transfers the resulting rational approximants to irrational parameters.  This
produces a fully explicit continued-fraction gap function for $f=1$ and, more
generally, for every polynomial $f$ with $f(0)\ne0$.  The argument also gives
the exact degree and leading coefficient of the corresponding
polynomial-vector fibre determinants.
\end{abstract}

\section{Introduction}

Fix $1<p<\infty$.  For $\alpha\in(0,1)$, the Bishop operator on
$L^p[0,1]$ is
\[
 (T_\alpha f)(t)=t f(\{t+\alpha\}),
\]
where $\{\cdot\}$ denotes fractional part.  A vector $f\in L^p[0,1]$ is
\emph{cyclic} for $T_\alpha$ if
\[
 \overline{\operatorname{span}}\{T_\alpha^n f:n\ge0\}
 =\overline{\{Q(T_\alpha)f:Q\in\C[\xi]\}}
 =L^p[0,1].
\]
The operator is called cyclic if it has at least one cyclic vector.  Thus
cyclicity asks whether polynomial combinations of the iterates of a single
function can approximate every element of $L^p[0,1]$; it is distinct from
hypercyclicity, which asks whether the orbit itself is dense.  We use the
standard terminology of linear dynamics; see
\cite{BayartMatheron,GrosseErdmannPeris}.

Bishop proposed this family in connection with the invariant-subspace
problem.  Early and subsequent work on invariant and hyperinvariant subspaces
for Bishop-type operators includes Davie~\cite{Davie},
MacDonald~\cite{MacDonald}, Flattot~\cite{Flattot}, and the more recent work of
Chamizo, Gallardo-Guti\'errez, Monsalve-L\'opez, and Ubis~\cite{ChamizoEtAl};
see also the monograph of Chalendar and Partington~\cite{ChalendarPartington}
for the broader operator-theoretic context.

B\'ehani studied Bishop operators from this point of view in~\cite{Behani}.
Building on the rational-parameter theory of Chalendar and Partington
\cite[Section~5.4]{ChalendarPartington}, and its multivariable extension by
Chalendar, Pozzi, and Partington~\cite[Section~4]{ChalendarPartingtonPozzi},
he recalled a fibrewise characterization of cyclic vectors for
$\alpha=r/q\in\Q$ and used it to prove, among other results, that every
function holomorphic near $[0,1]$ with $f(0)\ne0$ is cyclic for every rational
Bishop operator.  He then passed from rational to irrational parameters.  If
$(p_n/q_n)$ are the continued-fraction convergents of an irrational
$\alpha$, his Theorem~4.16 gives a function $\psi_f:\mathbb N\to\mathbb R_+$
such that sufficiently frequent gaps
\[
 q_{n+1}>\psi_f(q_n)
\]
imply that $f$ is cyclic for $T_\alpha$.  The proof obtains $\psi_f$
indirectly from finitely many rational approximants and continuity in the
parameter.  B\'ehani therefore asked in Question~5.4 whether $\psi_f$ can be
made explicit, already for the constant vector $f=1$.

We next explain why a finite determinant lies at the centre of both the
rational criterion and this quantitative question.  Let $\alpha=r/q$ with
$(r,q)=1$, and restrict $t$ to the fundamental interval
$I_q=[0,1/q)$.  The rotation $t\mapsto\{t+r/q\}$ has $q$ points in each
fibre.  If one seeks a representation
\[
 g=\sum_{j=0}^{q-1}h_jT_{r/q}^jf
\]
with $1/q$-periodic coefficient functions $h_j$, then evaluating at the $q$
points $t+\nu/q$ gives the linear system
\begin{equation}\label{eq:intro-fibre-system}
 \big(g(t+\nu/q)\big)_{\nu=0}^{q-1}
 =M^f_{r,q}(t)\big(h_j(t)\big)_{j=0}^{q-1},
 \qquad
 M^f_{r,q}(t)_{\nu j}
 =\big(T_{r/q}^jf\big)(t+\nu/q).
\end{equation}
The corresponding determinant
\[
 D^f_{r,q}(t)=\det M^f_{r,q}(t)
\]
is therefore not an auxiliary object: it is exactly the obstruction to
solving the fibre system.  The rational cyclicity theorem of Chalendar and
Partington, in the form used by B\'ehani, states that $f$ is cyclic for
$T_{r/q}$ precisely when this determinant is nonzero almost everywhere
\cite[Section~5.4]{ChalendarPartington}.  For analytic $f$ with $f(0)\ne0$,
it is enough to show that $D^f_{r,q}$ is not identically zero, which follows
from its explicit nonzero value at the origin.

For an \emph{effective} irrational-parameter argument, almost-everywhere
invertibility is not by itself sufficient.  Solving
\eqref{eq:intro-fibre-system} introduces $M^f_{r,q}(t)^{-1}$, whose entries
contain $1/D^f_{r,q}(t)$.  One must control the size and regularity of the
solutions $h_j$, approximate them by polynomials, and then quantify how the
resulting polynomial orbit changes when $r/q$ is replaced by a nearby
irrational parameter.  Small values of the determinant govern every one of
these estimates.  This is why B\'ehani singled out, for $f=1$, the possible
nonvanishing and monotonicity of the determinant on $I_q$ as a route toward
an explicit $\psi_1$.

For $f=1$, the entries in \eqref{eq:intro-fibre-system} are successive orbit
products.  After ordering the rows along the modular orbit and scaling
$z=qt$, the determinant becomes, up to a fixed nonzero factor and a row sign,
\[
 \D_{r,q}(z)=
 \det\left[\prod_{m=0}^{j-1}(z+c_{i+m})\right]_{0\le i,j\le q-1},
 \qquad c_i=ir\pmod q.
\]
This is the polynomial studied in the first part of the paper.  We show that
the proposed global nonvanishing route cannot work in general: for infinitely
many even denominators, $\D_{r,q}$ has a sign-changing zero in the positive
fundamental cell.  A resultant identity and discrete Fourier diagonalization
identify the responsible factor as the unique unpaired alternating Fourier
mode.  Odd denominators, in contrast, have only conjugate pairs of nontrivial
modes and hence a parity-based nonnegativity property.

The failure of global nonvanishing does not obstruct cyclicity.  Our
quantitative replacement uses the exact degree of $D_{r,q}$ and its nonzero
value at the origin.  Remez's inequality then bounds the measure of the set on
which the determinant is small.  We discard that set, invert the fibre matrix
on its complement, approximate the cutoff inverse by endpoint-corrected
Fej\'er polynomials, and use an explicit continuity modulus in $\alpha$.
This produces a completely explicit gap function for $f=1$, and the same
scheme applies to every polynomial $f$ with $f(0)\ne0$.

The determinant problem has recently been studied from a complementary
direction by Grivaux~\cite{Grivaux}.  His work concerns the polynomials
\[
 \mathcal P_n(s,X)=\sum_{k=0}^n s^{(k)}X^{n-k},
 \qquad
 s^{(k)}=s(s+1)\cdots(s+k-1),
\]
and a cyclic-product determinant $\theta_n(s)$.  He proves that all roots of
$\mathcal P_n(s,\cdot)$ are simple for $s>0$, obtains a precise localization
of these roots in angular sectors, and deduces that
\[
 (-1)^{n(n-1)/2}\theta_n(s)>0\qquad(s\ge0).
\]
The relation to the present paper is exact, but restricted to one orbit order:
as shown in Section~\ref{sec:grivaux},
\[
 \theta_{q-1}(z)=\D_{1,q}(z).
\]
Thus Grivaux proves strict nonvanishing for the consecutive-order subfamily
$r=1$.  In contrast, the modular order $c_i=ir\pmod q$ entering a general
rational Bishop operator changes the cyclic products inside the matrix and
cannot be removed by merely permuting its rows.  We show that this reordered
family has sign-changing positive zeros for infinitely many even denominators.
This does not conflict with Grivaux's theorem; it identifies exactly where the
two results diverge.

Our main contributions are as follows.  First, a resultant identity and
discrete Fourier diagonalization isolate the parity mechanism responsible for
positive zeros.  Second, we prove an infinite zero family by an entirely
analytic asymptotic argument.  Third,
we bypass these zeros by a quantitative Remez--Fej\'er construction and obtain
explicit gap functions first for $f=1$ and then for every polynomial $f$ with
$f(0)\ne0$.  Finally, a Fourier--Vandermonde calculation gives the exact degree
and leading coefficient of every polynomial-vector fibre determinant.  The
orbit determinant, its Fourier structure, and its zeros are treated first;
the quantitative cyclicity argument begins in
Section~\ref{sec:determinant-to-cyclicity}, where the scaled determinant is
returned to B\'ehani's fibre system.

\section{The orbit determinant}

Let \(q\ge2\) and \(1\le r<q\) be coprime.  Put
\[
 c_i=ir\pmod q,\qquad 0\le c_i\le q-1,
\]
and extend \(c_i\) periodically.  The integer-scaled, orbit-ordered fibre
determinant is
\begin{equation}\label{eq:determinant}
 \D_{r,q}(z)=
 \det\left[
  \prod_{m=0}^{j-1}(z+c_{i+m})
 \right]_{0\le i,j\le q-1}.
\end{equation}
An empty product is one.  Define
\begin{align}
 V_0(z)&=1,
 &
 V_i(z)&=\prod_{m=0}^{i-1}(z+c_m)\quad(1\le i\le q),\label{eq:V}\\
 F(z)&=V_q(z)=\prod_{a=0}^{q-1}(z+a),
 &
 \widetilde B_z(x)&=\sum_{i=0}^{q-1}V_i(z)x^i.\label{eq:Btilde}
\end{align}
The equality for \(F\) follows because \(c_0,\ldots,c_{q-1}\) is a
permutation of \(0,\ldots,q-1\).
Throughout the paper we write
\[
 \ell_q=\frac{(q-1)(q-2)}2.
\]

\section{The consecutive orbit and Grivaux's determinant}
\label{sec:grivaux}

We record the precise relation with~\cite{Grivaux}.  For
$a=(a_0,\ldots,a_n)$, let $M(a)$ be the $(n+1)\times(n+1)$ matrix whose
row starting at $a_i$ consists of the successive cyclic products of lengths
$0,1,\ldots,n$:
\[
 M(a)_{ij}=\prod_{m=0}^{j-1}a_{i+m},
 \qquad 0\le i,j\le n,
\]
where indices on $a$ are read modulo $n+1$.  Grivaux writes
\[
 \theta_n(s)=\det M(s,s+1,\ldots,s+n).
\]

\begin{proposition}[The common determinant subfamily]
\label{prop:grivaux-relation}
For every $q\ge2$ and every complex number $z$,
\begin{equation}\label{eq:theta-D}
 \D_{1,q}(z)=\theta_{q-1}(z).
\end{equation}
Moreover, if
\[
 \mathcal P_{q-1}(z,X)=\sum_{k=0}^{q-1}z^{(k)}X^{q-1-k},
\]
then
\begin{equation}\label{eq:grivaux-discriminant}
 \operatorname{disc}_X\mathcal P_{q-1}(z,X)
 =\left(\prod_{j=1}^{q-2}z^{(j)}\right)\D_{1,q}(z).
\end{equation}
In particular,
\begin{equation}\label{eq:r1-positive}
 (-1)^{\ell_q}\D_{1,q}(z)>0\qquad(z\ge0).
\end{equation}
\end{proposition}

\begin{proof}
When $r=1$, the orbit representatives are $c_i=i$ for
$0\le i\le q-1$.  The defining matrix in \eqref{eq:determinant} is therefore
exactly $M(z,z+1,\ldots,z+q-1)$, which proves \eqref{eq:theta-D}.
Formula \eqref{eq:grivaux-discriminant} is Proposition~3.3(i)
of~\cite{Grivaux}, with $n=q-1$ and $s=z$.  Finally, Corollary~5.5 of that
paper gives
\[
 (-1)^{(q-1)(q-2)/2}\theta_{q-1}(z)>0
 \quad(z\ge0),
\]
which is \eqref{eq:r1-positive}.
\end{proof}

\begin{remark}[Complementarity of the two approaches]
Grivaux obtains \eqref{eq:r1-positive} by relating the determinant to a
discriminant and proving simplicity and angular localization of the roots of
$\mathcal P_n(s,\cdot)$.  The resultant viewpoint is common to both works.
The present Fourier factorization addresses the additional modular ordering
$i\mapsto ir\pmod q$ and shows that, for even $q$, its unpaired alternating
character can produce sign-changing zeros.  In particular, the zero examples
below necessarily have $r\ne1$ and do not contradict
\eqref{eq:r1-positive}.  Conversely, our parity factorization gives only
nonnegativity for general odd $q$ and does not recover the strict
nonvanishing, root localization, or Hurwitz-stability questions studied
in~\cite{Grivaux}.
\end{remark}

\section{Resultant and Fourier factorization}

\begin{theorem}[Resultant identity]\label{thm:resultant}
In \(\mathbb Z[z]\),
\begin{equation}\label{eq:resultant}
 \Res_x\!\left(F(z)x^q-1,\widetilde B_z(x)\right)
 =
 (-1)^{\ell_q}\D_{r,q}(z)
 \prod_{i=0}^{q-1}V_i(z).
\end{equation}
\end{theorem}

\begin{proof}
It is enough to prove the identity for \(z>0\), since both sides are
polynomials in \(z\).  Let \(s=F(z)^{1/q}>0\), and put
\[
 a_i=\frac{V_i(z)}{s^i}\qquad(0\le i\le q-1).
\]
The extension \(V_{i+q}=F V_i\) makes \(a_i\) \(q\)-periodic.  The
\((i,j)\)-entry of the matrix in \eqref{eq:determinant} is
\[
 \frac{V_{i+j}}{V_i}=s^j\frac{a_{i+j}}{a_i}.
\]
Consequently, if \(C=[a_{i+j}]_{0\le i,j<q}\), then
\begin{equation}\label{eq:detC}
 \D_{r,q}(z)
 =
 \frac{F(z)^{q-1}}{\prod_{i=0}^{q-1}V_i(z)}\det C.
\end{equation}

Let \(R=[a_{j-i}]_{0\le i,j<q}\), with indices modulo \(q\), and let \(J\)
be the permutation matrix induced by \(i\mapsto-i\pmod q\).  Then
\(C=JR\), and
\[
 \det J=(-1)^{\ell_q}.
\]
If \(\omega=e^{2\pi i/q}\) and
\[
 B_z(\xi)=\sum_{i=0}^{q-1}a_i\xi^i
          =\widetilde B_z(\xi/s),
\]
Fourier diagonalization of \(R\) gives
\[
 \det R=\prod_{k=0}^{q-1}B_z(\omega^k)
       =\prod_{k=0}^{q-1}\widetilde B_z(\omega^k/s).
\]
The numbers \(\omega^k/s\) are the roots of \(F(z)x^q-1\).  By the
defining product formula for the resultant,
\[
 \Res_x(Fx^q-1,\widetilde B_z)
 =
 F^{q-1}\prod_{k=0}^{q-1}\widetilde B_z(\omega^k/s).
\]
Combining this with \eqref{eq:detC} proves \eqref{eq:resultant}.
\end{proof}

The identity yields a real factorization with an immediate parity
consequence.  Put
\[
 P_{r,q}(z)=(-1)^{\ell_q}\D_{r,q}(z),
 \qquad
 \rho(z)=\frac{F(z)^{q-1}}{\prod_{i=0}^{q-1}V_i(z)}>0
 \quad(z>0).
\]

\begin{corollary}[Parity factorization]\label{cor:parity}
For \(z>0\), the following statements hold.

\begin{enumerate}
\item If \(q\) is odd, then
\begin{equation}\label{eq:odd}
 P_{r,q}(z)
 =
 \rho(z)B_z(1)
 \prod_{k=1}^{(q-1)/2}|B_z(\omega^k)|^2
 \ge0.
\end{equation}
Hence every positive real zero of \(\D_{r,q}\) has even multiplicity, and
the determinant cannot change sign there.

\item If \(q\) is even, then
\begin{equation}\label{eq:even}
 P_{r,q}(z)
 =
 \rho(z)B_z(1)B_z(-1)
 \prod_{k=1}^{q/2-1}|B_z(\omega^k)|^2,
\end{equation}
where
\begin{equation}\label{eq:alternating}
 B_z(-1)=
 \sum_{i=0}^{q-1}(-1)^i\frac{V_i(z)}{F(z)^{i/q}}.
\end{equation}
Every sign-changing zero must therefore come from the real alternating
factor \(B_z(-1)\).  Zeros contributed only by nonreal Fourier modes have
even multiplicity.
\end{enumerate}
\end{corollary}

\begin{proof}
The coefficients \(a_i\) are real and positive, so \(B_z(1)>0\), while
\[
 B_z(\omega^{q-k})=\overline{B_z(\omega^k)}.
\]
Pairing the conjugate factors in the proof of
Theorem~\ref{thm:resultant} gives \eqref{eq:odd} when \(q\) is odd.
For even \(q\), the additional unpaired root of unity is
\(\omega^{q/2}=-1\), which gives \eqref{eq:even} and
\eqref{eq:alternating}.  The multiplicity statements follow from the
modulus-square factors.
\end{proof}

\begin{remark}[Why the parity is intrinsic]
When \(q\) is even, coprimality forces \(r\) to be odd.  The orbit
\(c_i=ir\pmod q\) therefore alternates between even and odd residues.  The
factor \(B_z(-1)\) is precisely the Fourier character that records this
alternation.  There is no corresponding real nontrivial character for odd
\(q\).
\end{remark}

\section{A concrete sign-changing zero}

The exact integer calculation at \((q,r)=(16,9)\) can now be interpreted
structurally.

\begin{proposition}\label{prop:q16}
For \((q,r)=(16,9)\), the alternating factor in
\eqref{eq:alternating} vanishes at some
\[
 z_*\in(11/16,3/4).
\]
Consequently,
\[
 \D_{9,16}(z_*)=0,
 \qquad
 z_*/16\in(11/256,3/64)
\]
in the original Bishop fundamental interval.
\end{proposition}

\begin{proof}
We give an exact rational certificate for the two required signs.  This also
makes the finite computation reproducible without relying on a black-box
determinant evaluation.

For $z=a/b>0$, with $a,b$ positive integers, put
\[
 W_0=1,\qquad
 W_i=\prod_{m=0}^{i-1}(a+bc_m)\quad(1\le i\le16),
 \qquad x=W_{16}^{1/16}.
\]
Since $F(a/b)=b^{-16}W_{16}$, one has
\[
 \frac{V_i(a/b)}{F(a/b)^{i/16}}=\frac{W_i}{x^i}.
\]
Consequently the alternating factor has the same sign as the polynomial
\begin{equation}\label{eq:q16-Q}
 Q_{a,b}(x):=x^{15}B_{a/b}(-1)
 =\sum_{i=0}^{15}(-1)^iW_i x^{15-i}.
\end{equation}
Here the modular orbit is
\[
 (c_0,\ldots,c_{15})
 =(0,9,2,11,4,13,6,15,8,1,10,3,12,5,14,7).
\]

For completeness, we now certify the sign of \eqref{eq:q16-Q} using only
integer arithmetic.  If $0<L<x<U$, then the positivity of all $W_i$
gives
\begin{align}
 \underline Q_{a,b}(L,U)
 &:=\sum_{\substack{0\le i\le15\\ i\ {\rm even}}}
       W_iL^{15-i}
   -\sum_{\substack{0\le i\le15\\ i\ {\rm odd}}}
       W_iU^{15-i}
 \le Q_{a,b}(x),\label{eq:q16-lower}\\
 Q_{a,b}(x)&\le
 \overline Q_{a,b}(L,U)
 :=\sum_{\substack{0\le i\le15\\ i\ {\rm even}}}
       W_iU^{15-i}
   -\sum_{\substack{0\le i\le15\\ i\ {\rm odd}}}
       W_iL^{15-i}.\label{eq:q16-upper}
\end{align}
The following table records $W_{16}$ and rational brackets for $x$.  Each
bracket is verified by raising its endpoints to the sixteenth power and
comparing with $W_{16}$.  Substitution in
\eqref{eq:q16-lower}--\eqref{eq:q16-upper}, followed by clearing the
denominator $1000^{15}$, gives the polynomial bounds displayed below.
\[
\begin{array}{c|c|c}
 (a,b)&W_{16}&(L,U)\\ \hline
 (11,16)&
 122274082645535806376391906530625&
 (101264/1000,101265/1000)\\[2pt]
 (3,4)&
 36453104912477522894625&
 (25710/1000,25711/1000)
\end{array}
\]
For these two rows, respectively, direct substitution in
\eqref{eq:q16-lower}--\eqref{eq:q16-upper} gives the exact rational bounds
\begin{align*}
 9\cdot10^{26}
 &<\underline Q_{11,16}<Q_{11,16}(x)
 <\overline Q_{11,16}<2\cdot10^{27},\\
 -5\cdot10^{18}
 &<\underline Q_{3,4}<Q_{3,4}(x)
 <\overline Q_{3,4}<-2\cdot10^{18}.
\end{align*}
Every displayed assertion is a comparison of integers after multiplication
by $1000^{16}$ for the root brackets or by $1000^{15}$ for the polynomial
bounds.

It follows that
\[
 B_{11/16}(-1)>0,
 \qquad B_{3/4}(-1)<0.
\]
The function $z\mapsto B_z(-1)$ is continuous for $z>0$, because
$F(z)>0$ there and the positive sixteenth root is continuous.  Hence it
vanishes at some $z_*\in(11/16,3/4)$.  Formula \eqref{eq:even} gives
$\D_{9,16}(z_*)=0$.  Finally, the unscaled Bishop variable is $t=z/16$,
which yields $t\in(11/256,3/64)$.
\end{proof}

Thus the \(q=16\) counterexample is not an unexplained cancellation among
complex modes.  It is caused by the unique real alternating mode that exists
only for even denominator.

\section{An analytic three-block family}

We now give a fully analytic infinite-family argument.  It yields all
sufficiently large members, without asserting an explicit starting index.  Put
\begin{equation}\label{eq:q3}
 \widehat q_n=6n-2,\qquad \widehat r=3,
\end{equation}
and denote the corresponding quantities in
\eqref{eq:V}--\eqref{eq:alternating} by hats.  Thus
\begin{equation*}
 \widehat s_n=(\widehat q_n!)^{1/\widehat q_n},\qquad
 \widehat x_n=\frac3{\widehat s_n},\qquad
 \widehat B_n=\widehat B_1^{(n)}(-1).
\end{equation*}

For $\nu,x>0$, define
\begin{equation}\label{eq:J}
 J_\nu(x)=\frac1{\Gamma(\nu)}
 \int_0^\infty\frac{t^{\nu-1}e^{-t}}{1+xt}\,dt.
\end{equation}
We shall use the following elementary finite-sum identity.  If
$\alpha,x>0$ and $L\ge1$ is an integer, then
\begin{equation}\label{eq:block-integral}
 \sum_{j=0}^{L-1}(-1)^j(\alpha)_j x^j
 =J_\alpha(x)-(-1)^Lx^L(\alpha)_LJ_{\alpha+L}(x).
\end{equation}
Indeed, integrate the identity
\[
 \sum_{j=0}^{L-1}(-xt)^j
 =\frac{1-(-xt)^L}{1+xt}
\]
against $t^{\alpha-1}e^{-t}/\Gamma(\alpha)$ and use
$\Gamma(\alpha+L)/\Gamma(\alpha)=(\alpha)_L$.

\begin{theorem}[Three-block family]\label{thm:threeblock}
For every sufficiently large integer \(n\),
\begin{equation}\label{eq:three-negative}
 \widehat B_n<0.
\end{equation}
Consequently, \(\D_{3,6n-2}\) has a sign-changing zero in \(z\in(0,1)\)
for every sufficiently large \(n\).  More precisely,
\begin{equation}\label{eq:three-limit}
 \lim_{n\to\infty}\widehat B_n
 =\frac{e}{e+1}(1-\widehat A)<0,
 \qquad
 \widehat A=
 \frac{e^{1/3}(2\pi)^{2/3}}
 {3^{1/3}\Gamma(1/3)}>1.
\end{equation}
\end{theorem}

\begin{proof}
The order of the weights \(\widehat w_i=1+(3i\bmod\widehat q_n)\)
has only three arithmetic blocks:
\begin{equation}\label{eq:three-block-table}
\begin{array}{c|c|c|c}
h&\widehat I_{n,h}\le i<\widehat I_{n,h+1}
 &\widehat w_i&(-1)^{\widehat I_{n,h}}\\ \hline
0&0\le i<2n&1,4,\ldots,6n-2&+1\\
1&2n\le i<4n-1&3,6,\ldots,6n-3&+1\\
2&4n-1\le i<6n-2&2,5,\ldots,6n-4&-1.
\end{array}
\end{equation}
Let \(\widehat A_{n,h}=\widehat a_{n,\widehat I_{n,h}}\), with
\(\widehat A_{n,0}=\widehat A_{n,3}=1\).  Direct multiplication along
the blocks gives
\begin{align}
 \widehat A_{n,1}
 &=\frac{3^{2n}\Gamma(2n+1/3)}
 {\Gamma(1/3)\Gamma(6n-1)^{2n/\widehat q_n}},\label{eq:three-A1}\\
 \widehat A_{n,2}
 &=\widehat A_{n,1}
 \frac{3^{2n-1}\Gamma(2n)}
 {\Gamma(6n-1)^{(2n-1)/\widehat q_n}}.\label{eq:three-A2}
\end{align}
Using the finite-sum identity \eqref{eq:block-integral} on the three blocks
and collecting adjacent boundary terms gives the exact reduction
\begin{align}
 \widehat B_n={}&
 \bigl(J_{1/3}(\widehat x_n)-J_{2n-1/3}(\widehat x_n)\bigr)\nonumber\\
 &+\widehat A_{n,1}
 \bigl(J_1(\widehat x_n)-J_{2n+1/3}(\widehat x_n)\bigr)\nonumber\\
 &-\widehat A_{n,2}
 \bigl(J_{2/3}(\widehat x_n)-J_{2n}(\widehat x_n)\bigr).
 \label{eq:three-reduction}
\end{align}

Stirling's formula yields
\begin{equation}\label{eq:three-asymptotics}
 \widehat x_n\longrightarrow0,\qquad
 2n\widehat x_n\longrightarrow e,\qquad
 \widehat A_{n,1}=O(n^{-1/3}),\qquad
 \widehat A_{n,2}\longrightarrow\widehat A.
\end{equation}
Indeed, retaining the constant terms in \eqref{eq:three-A2} gives
\begin{equation*}
 \log\widehat A
 =\frac13+\frac23\log(2\pi)
  -\frac13\log3-\log\Gamma(1/3).
\end{equation*}
For fixed $\nu$, dominated convergence in \eqref{eq:J} gives
\(J_\nu(\widehat x_n)\to1\).  If $\nu_n=2n+O(1)$ and
$T_{\nu_n}$ has the Gamma distribution with shape $\nu_n$ and unit scale,
then
\[
 \mathbb E\left|\frac{T_{\nu_n}}{\nu_n}-1\right|^2
 =\frac1{\nu_n}\longrightarrow0.
\]
Together with $\nu_n\widehat x_n\to e$, boundedness and convergence in
probability give
\begin{equation*}
 J_{2n+O(1)}(\widehat x_n)\longrightarrow\frac1{1+e}.
\end{equation*}
Substitution in \eqref{eq:three-reduction} proves the equality in
\eqref{eq:three-limit}.

It remains only to check the strict sign, for which no interval computation
is needed.  Since \(e^{-y}\le1-y/2\) on \([0,1]\) and
\(y^{-2/3}\le1\) on \([1,\infty)\),
\begin{equation*}
 \Gamma(1/3)
 <\int_0^1y^{-2/3}(1-y/2)\,dy
  +\int_1^\infty e^{-y}\,dy
 =\frac{21}{8}+e^{-1}<3.
\end{equation*}
Here $e>8/3$ proves the last inequality.  The elementary bounds $e>8/3$ and
$\pi>3$ also give $e(2\pi)^2>96>81$, and hence
\begin{equation*}
 e^{1/3}(2\pi)^{2/3}>3^{4/3}
 >3^{1/3}\Gamma(1/3),
\end{equation*}
so \(\widehat A>1\).  The limit in
\eqref{eq:three-limit} is strictly negative, proving
\eqref{eq:three-negative} for all sufficiently large \(n\).

Finally, fix $n$ and let $z\downarrow0$.  For every $1\le i<\widehat q_n$,
$\widehat V_i(z)$ contains the factor $z$, whereas
\[
 \widehat F(z)=z\prod_{a=1}^{\widehat q_n-1}(z+a)\asymp z.
\]
Hence
$\widehat V_i(z)/\widehat F(z)^{i/\widehat q_n}
=O(z^{1-i/\widehat q_n})\to0$, and therefore
\(\widehat B_z^{(n)}(-1)\to1\).  Continuity and
Corollary~\ref{cor:parity} give a sign-changing zero in \((0,1)\).
\end{proof}

\section{From determinant zeros to explicit cyclicity}
\label{sec:determinant-to-cyclicity}

The preceding factorization settles the structural question that motivates
the quantitative part of the paper.  Odd denominators enjoy a parity-based
nonnegativity property, but even denominators admit genuine sign-changing
zeros, and these occur in an infinite family.  Therefore a proof of explicit
cyclicity criteria cannot rely on a uniform lower bound for the fibre
determinant over the whole fundamental interval.  The substitute is to use
the exact degree and the nonzero value at one point to control, by Remez's
inequality, the measure of the region where the determinant is small.  After
discarding that region, the inverse matrix is quantitatively well behaved.

We now implement this strategy first for the constant vector and then for all
polynomial vectors with nonzero constant term.

\section{Explicit gap functions: statement and constants}\label{sec:gap-statement}

Fix $1<p<\infty$.  For every integer $q\geq2$, put
\begin{align*}
 s_q&=\frac{q(q-1)}2,
 &\ell_q&=\frac{(q-1)(q-2)}2,
 &\varepsilon_q&=\frac1{q+1},
 &b_q&=q^{-s_q},
 &W_q&=\frac{q!}{q^q},\\
 \eta_q&=
 \begin{cases}
  b_q/2,&q=2,\\[2pt]
  \displaystyle\frac{b_q}{2}
  \left(\frac{\varepsilon_q^p}{8}\right)^{\ell_q},&q\ge3,
 \end{cases}
 &a_q&=\frac{\varepsilon_q^p}{8q},
 &H_q&=\frac{q!}{\eta_q},
 &\Gamma_q&=3+\frac4{W_q}.
\end{align*}
Define
\begin{equation}\label{eq:Lq}
 L_q=
 \frac{H_q}{a_q}
 +\frac{q!\,q(q-1+\pi)}{\eta_q}
 +\frac{2(q!)^2q^2}{\eta_q^2},
\end{equation}
and then
\begin{align}
 \Lambda_q&=\max\left\{1,\frac{\pi qL_q}{\varepsilon_q}\right\},
 &\qquad N_q&=\left\lceil4\Lambda_q
          \log(4\Lambda_q+\mathrm e)\right\rceil,\label{eq:N}\\
 d_q&=qN_q+q-1,\label{eq:N-d}\\
 C_q&=3qH_qN_q\Gamma_q^{N_q},\label{eq:Cq}\\
 \delta_q&=\min\left\{
   \frac{1}{2d_q},
   \frac{\varepsilon_q^p}{4C_q^p d_q^{p+1}}
 \right\}.\label{eq:deltaq}
\end{align}
Finally set
\begin{equation}\label{eq:psi}
 \boxed{
 \psi_{1,p}(1)=1,\qquad
 \psi_{1,p}(q)=\frac1{q\delta_q}\quad(q\ge2).
 }
\end{equation}
Equivalently, for $q\ge2$,
\[
 \psi_{1,p}(q)=
 \max\left\{
  \frac{2d_q}{q},
  \frac{4C_q^p d_q^{p+1}}{q\varepsilon_q^p}
 \right\}.
\]

\begin{theorem}\label{thm:main}
Let $\alpha\in(0,1)\setminus\Q$, and let $(p_n/q_n)$ be its continued-fraction
convergents.  If for every $n$ there is $n_0\ge n$ such that
\[
 q_{n_0+1}>\psi_{1,p}(q_{n_0}),
\]
then the constant function $1$ is cyclic for $T_\alpha$ on $L^p[0,1]$.
\end{theorem}

The constants are deliberately crude.  Formula \eqref{eq:psi}, rather than its
size, is the point of the result.  The proof only uses
$0<\varepsilon_q\le1$, $\varepsilon_q\to0$; the choice
$\varepsilon_q=(q+1)^{-1}$ is a simple explicit polynomial scale.

\section{Quantitative rational-parameter approximation}\label{sec:rational-approx}

Enumerate the integers by
\[
 m_1=0,\qquad m_{2k}=k,\qquad m_{2k+1}=-k\quad(k\ge1),
\]
and put $g_n(x)=e^{2\pi i m_nx}$.  The closed linear span of the $g_n$ is
$L^p[0,1]$.

Fix $q\ge2$, $1\le r<q$ with $(r,q)=1$, and $n\le q$.  On the fundamental
interval $I_q=[0,1/q]$, define the $q\times q$ fibre matrix
\begin{equation}\label{eq:M}
 M_{r,q}(t)_{\nu j}
   =\big(T_{r/q}^{j}1\big)(t+\nu/q),
 \qquad 0\le\nu,j\le q-1,
\end{equation}
where endpoint values are interpreted as one-sided limits.  On the interior of
$I_q$, every entry of $M_{r,q}$ is a product of at most $q-1$ affine factors
$t+c/q$, with $0\le c\le q-1$.  Hence
\begin{equation}\label{eq:M-bounds}
 |M_{r,q}(t)_{\nu j}|\le1,
 \qquad
 \left|\frac{d}{dt}M_{r,q}(t)_{\nu j}\right|\le q.
\end{equation}
Let $D_{r,q}(t)=\det M_{r,q}(t)$.  This is a real polynomial on $I_q$.
It is the unscaled version of the orbit determinant studied above.  More
precisely, if $\epsilon_{r,q}\in\{\pm1\}$ is the sign of the row
permutation $i\mapsto ir\pmod q$, then
\begin{equation}\label{eq:scaled-relation}
 \mathcal D_{r,q}(qt)=\epsilon_{r,q}q^{s_q}D_{r,q}(t),
 \qquad s_q=\frac{q(q-1)}2.
\end{equation}
Thus the parity and zero results concern exactly the same determinant used in
the approximation argument below.
The determinant calculation used in the proof of B\'ehani's rational cyclicity
theorem gives
\begin{equation}\label{eq:D0}
 |D_{r,q}(0)|
 =\prod_{k=1}^{q-1}\left\{\frac{kr}{q}\right\}^{k}
 \ge q^{-s_q}=b_q.
\end{equation}

\begin{lemma}\label{lem:exact-degree}
The polynomial $D_{r,q}$ has exact degree
\[
 \ell_q=\frac{(q-1)(q-2)}2,
\]
and its leading coefficient has modulus $1/q$.
\end{lemma}

\begin{proof}
Put $c_\nu=\nu/q$, $c=(c_\nu)_{\nu=0}^{q-1}$, and let $P_r$ be the cyclic
permutation
\[
 (P_ru)_\nu=u_{\nu+r\ ({\rm mod}\ q)}.
\]
On the interior of $I_q$, the columns of $M_{r,q}(t)$ are
\[
 \mathbf 1,A(t)\mathbf 1,\ldots,A(t)^{q-1}\mathbf 1,
 \qquad A(t)=(tI+\operatorname{diag}c)P_r.
\]
Since $A(t)\mathbf 1=t\mathbf 1+c$, subtracting $t$ times column $j-1$
from column $j$, successively for $j=q-1,q-2,\ldots,1$, gives
\begin{equation}\label{eq:column-reduction}
 D_{r,q}(t)=
 \det[\mathbf 1,c,A(t)c,\ldots,A(t)^{q-2}c].
\end{equation}
The right side has degree at most $\ell_q$, and its coefficient at
$t^{\ell_q}$ is
\begin{equation}\label{eq:leading-det}
 L_{r,q}=\det[\mathbf 1,c,P_rc,\ldots,P_r^{q-2}c].
\end{equation}

For completeness, diagonalize $P_r$ by the unitary discrete Fourier transform.
The nonconstant Fourier coefficients of $c$ are, up to a permutation,
\[
 \widehat c_m=-\frac{1}{\sqrt q(1-\omega^m)},
 \qquad 1\le m\le q-1,
 \qquad \omega=e^{2\pi i/q}.
\]
Expanding \eqref{eq:leading-det} along the constant Fourier mode leaves the
Vandermonde determinant on the $q-1$ nontrivial $q$th roots of unity.  Using
\[
 \prod_{m=1}^{q-1}|1-\omega^m|=q,
 \qquad
 \left|\operatorname{Vand}(\omega,\ldots,\omega^{q-1})\right|
 =q^{(q-2)/2},
\]
one obtains
\[
 |L_{r,q}|=\sqrt q\,
 q^{-(q-1)/2}q^{-1}q^{(q-2)/2}=\frac1q.
\]
Thus $L_{r,q}\ne0$, proving both assertions.
\end{proof}

\subsection*{The small-determinant set}

We use the following standard measurable-set form of the Remez inequality
\cite{Remez,BorweinErdelyi}: if $P$ is a real polynomial of degree at most
$s$, $I$ is an interval and $E\subset I$ is measurable with positive measure,
then
\begin{equation}\label{eq:remez}
 \sup_I|P|\le\left(\frac{4|I|}{|E|}\right)^s\sup_E|P|.
\end{equation}
For $q\ge3$, apply this to
\[
 E_{r,q}=\{t\in I_q:|D_{r,q}(t)|<2\eta_q\}.
\]
Equations \eqref{eq:D0}, Lemma \ref{lem:exact-degree}, and
\eqref{eq:remez} imply
\[
 |E_{r,q}|
 \le \frac4q\left(\frac{2\eta_q}{b_q}\right)^{1/\ell_q}
 =\frac{\varepsilon_q^p}{2q}.
\]
For $q=2$, the determinant is constant of modulus $b_q$, so the same conclusion
holds because $E_{r,q}$ is empty.
The absolute value of the determinant is $1/q$-periodic after the fibre rows
are permuted.  Thus the corresponding bad set in $[0,1]$ has measure at most
$\varepsilon_q^p/2$.

Let $\chi_q:[0,\infty)\to[0,1]$ be the piecewise-linear function which is zero
on $[0,\eta_q]$, one on $[2\eta_q,\infty)$, and affine in between.  Let
$\rho_q:I_q\to[0,1]$ be zero on
\[
 [0,a_q]\cup[1/q-a_q,1/q],
\]
one on $[2a_q,1/q-2a_q]$, and affine on the two intervening intervals.
The set where $\rho_q\ne1$, repeated over the $q$ fundamental intervals, has
measure at most $4qa_q=\varepsilon_q^p/2$.

For $t\in I_q$ let
\[
 G_n(t)=\big(g_n(t+\nu/q)\big)_{\nu=0}^{q-1}.
\]
Where $D_{r,q}(t)\ne0$, define
\begin{equation}\label{eq:h}
 h(t)=\rho_q(t)\chi_q(|D_{r,q}(t)|)
       M_{r,q}(t)^{-1}G_n(t),
\end{equation}
and set $h(t)=0$ when $D_{r,q}(t)=0$.  The cutoff makes this a continuous
vector-valued function, and the endpoint cutoff makes every component vanish
at both endpoints of $I_q$.

If the components of $h$ are extended $1/q$-periodically, then on $[0,1]$
\begin{equation}\label{eq:fibre-exact}
 \sum_{j=0}^{q-1}h_j T_{r/q}^{j}1
 =\rho_q\chi_q(|D_{r,q}|)g_n.
\end{equation}
The right side differs from $g_n$ only on a set of measure at most
$\varepsilon_q^p$.  Consequently
\begin{equation}\label{eq:cutoff-error}
 \norm{\rho_q\chi_q(|D_{r,q}|)g_n-g_n}_p\le\varepsilon_q.
\end{equation}

\subsection*{Explicit Lipschitz and Fej\'er bounds}

Every cofactor of $M_{r,q}$ has absolute value at most $(q-1)!$.  Differentiating
the Leibniz formula and using \eqref{eq:M-bounds} gives
\begin{align*}
 |D_{r,q}'|&\le q!q^2,\\
 |(\operatorname{adj}M_{r,q})'_{j\nu}|
   &\le(q-1)!q(q-1).
\end{align*}
Since $n\le q$, one has $|m_n|\le q/2$ and $|g_n'|\le\pi q$.  If
$u=(\operatorname{adj}M_{r,q})G_n$, it follows that
\[
 |u_j|\le q!,
 \qquad
 |u_j'|\le q!q(q-1+\pi).
\]
Using $\operatorname{Lip}(\chi_q)\le1/\eta_q$ and
$\operatorname{Lip}(\rho_q)\le1/a_q$, formula \eqref{eq:h} therefore gives
\begin{equation}\label{eq:h-bounds}
 |h_j|\le H_q,
 \qquad
 \operatorname{Lip}(h_j)\le L_q,
\end{equation}
with $L_q$ as in \eqref{eq:Lq}.

For $t\in I_q$ put
\[
 w_q(t)=\prod_{k=0}^{q-1}(t+k/q).
\]
This maps $I_q$ increasingly onto $[0,W_q]$ and
\[
 w_q'(t)\ge w_q'(0)=\frac{(q-1)!}{q^{q-1}}=W_q.
\]
Thus $F_j=h_j\circ w_q^{-1}$ is continuous on $[0,W_q]$, vanishes at both
endpoints, is bounded by $H_q$, and has Lipschitz constant at most $L_q/W_q$.

Define the even $2\pi$-periodic function
\[
 \Phi_j(\theta)=F_j\!\left(\frac{W_q}{2}(1+\cos\theta)\right).
\]
It is bounded by $H_q$ and has Lipschitz constant at most $L_q/2$.  We use the
standard Fej\'er summability normalization (see, for example,
\cite{Katznelson}).  Let $\sigma_{N_q}\Phi_j$ be its Fej\'er mean of order
$N_q$.  If
\[
 K_N(u)=\frac1N\left(\frac{\sin(Nu/2)}{\sin(u/2)}\right)^2,
\]
then $K_N(u)\le\min\{N,\pi^2/(Nu^2)\}$ for $|u|\le\pi$.  Consequently
\begin{equation}\label{eq:Fejer-error}
 \|\sigma_N\Phi_j-\Phi_j\|_\infty
 \le \frac{L_q}{4\pi}\int_{-\pi}^{\pi}|u|K_N(u)\,du
 \le\frac{\pi L_q(\log N+1/2)}{2N}.
\end{equation}
The Fej\'er mean is even, hence it is an algebraic polynomial $P_j$ of degree
at most $N_q-1$ in $s=W_q(1+\cos\theta)/2$.  Correct its endpoint values by
\begin{equation}\label{eq:endpoint-correction}
 B_j(s)=P_j(s)-P_j(0)\left(1-\frac{s}{W_q}\right)
                 -P_j(W_q)\frac{s}{W_q}.
\end{equation}
Since $F_j(0)=F_j(W_q)=0$, this doubles at most the error in
\eqref{eq:Fejer-error}.  The definition \eqref{eq:N} and the elementary bound
\[
 \log N_q+\frac12\le4\log(4\Lambda_q+\mathrm e)
\]
therefore give
\begin{equation}\label{eq:Fejer-uniform}
 \|B_j-F_j\|_\infty\le\frac{\varepsilon_q}{q}.
\end{equation}
Now define the completely explicit polynomial
\begin{equation}\label{eq:Q}
 Q_{r,q,n}(\xi)=\sum_{j=0}^{q-1}\xi^jB_j(\xi^q).
\end{equation}
On each half-open cell $[\nu/q,(\nu+1)/q)$, the multiplier of
$T_{r/q}^q$ is $w_q(t)$, with $t=x-\nu/q$.  It jumps from $W_q$ to zero
at a cell boundary.  This causes no mismatch because
$B_j(0)=B_j(W_q)=0$ by \eqref{eq:endpoint-correction}.  Thus $B_j(w_q)$
has compatible one-sided values at every boundary.  Since $T_{r/q}^q$ is multiplication by this cellwise
$1/q$-periodic function,
\eqref{eq:fibre-exact}--\eqref{eq:Q} imply
\begin{equation}\label{eq:rational-approx}
 \norm{Q_{r,q,n}(T_{r/q})1-g_n}_p\le2\varepsilon_q.
\end{equation}

The degree of $Q_{r,q,n}$ is at most $d_q$.  To bound its coefficients, write
the Fej\'er mean in the Chebyshev basis:
\[
 P_j(s)=\widehat\Phi_j(0)+2\sum_{k=1}^{N_q-1}
 \left(1-\frac{k}{N_q}\right)\widehat\Phi_j(k)
 T_k(2s/W_q-1).
\]
Here $|\widehat\Phi_j(k)|\le H_q$.  If
$A_k=\|T_k(2s/W_q-1)\|_{\ell^1}$, then $A_0=1$,
$A_1=1+2/W_q$, and the Chebyshev recurrence gives
\[
 A_{k+1}\le2(1+2/W_q)A_k+A_{k-1}.
\]
It follows inductively that
\[
 A_k\le\Gamma_q^k,
 \qquad \Gamma_q=3+4/W_q.
\]
The Fej\'er operator is positive and preserves constants, so
$|P_j(0)|,|P_j(W_q)|\le H_q$.  Formula
\eqref{eq:endpoint-correction} now yields
\[
 \|B_j\|_{\ell^1}\le3H_qN_q\Gamma_q^{N_q}.
\]
Summing over $j$ gives
\begin{equation}\label{eq:coeff}
 \sum_k|[\xi^k]Q_{r,q,n}|
 \le3qH_qN_q\Gamma_q^{N_q}=C_q.
\end{equation}
All estimates are independent of $r$ and of $n\le q$.

\section{An explicit continuity modulus in the parameter}

Write
\[
 P_{k,\alpha}=T_\alpha^k1
 =\prod_{j=0}^{k-1}\{x+j\alpha\}.
\]
If $h=|\alpha-\beta|$ and $kh\le1$, then
\[
 \norm{\{x+j\alpha\}-\{x+j\beta\}}_p^p
 =jh(1-jh)^p+(1-jh)(jh)^p\le2jh.
\]
Telescoping the products and applying Minkowski's inequality yields
\begin{equation}\label{eq:power-continuity}
 \norm{P_{k,\alpha}-P_{k,\beta}}_p
 \le2^{1/p}h^{1/p}k^{1+1/p}.
\end{equation}
Consequently, if $Q$ has degree at most $d_q$ and coefficient $\ell^1$-norm at
most $C_q$, then, whenever $d_qh\le1$,
\begin{equation}\label{eq:Q-continuity}
 \norm{Q(T_\alpha)1-Q(T_\beta)1}_p
 \le2^{1/p}h^{1/p}C_qd_q^{1+1/p}.
\end{equation}
By \eqref{eq:deltaq}, the right side is strictly smaller than $\varepsilon_q$
when $h<\delta_q$.

\section{Proof of Theorem \ref{thm:main}}

Fix $n$.  By hypothesis, arbitrarily far out in the convergent sequence there is
an index $k$ with
\[
 q_{k+1}>\psi_{1,p}(q_k)=\frac1{q_k\delta_{q_k}}.
\]
For all sufficiently large such $k$, one has $n\le q_k$.  The continued-fraction
estimate gives
\[
 \left|\alpha-\frac{p_k}{q_k}\right|
 <\frac1{q_kq_{k+1}}<\delta_{q_k}.
\]
Use the polynomial $Q_{p_k,q_k,n}$ from \eqref{eq:Q}.  Combining
\eqref{eq:rational-approx} and \eqref{eq:Q-continuity},
\[
 \norm{Q_{p_k,q_k,n}(T_\alpha)1-g_n}_p
 <3\varepsilon_{q_k}=\frac3{q_k+1}.
\]
Along these indices the right side tends to zero.  Hence every $g_n$ belongs to
the closed cyclic subspace generated by $1$ under $T_\alpha$.  The trigonometric
system is total in $L^p[0,1]$, so that cyclic subspace is all of $L^p[0,1]$.
\qed

\section{Extension to polynomial vectors}

The same argument can be made effective for every polynomial vector which is
covered by B\'ehani's rational cyclicity theorem.  We give all constants because
the treatment of the jump of the periodic extension of $f$ in the parameter
continuity estimate is essential.

Let
\[
 f(z)=\sum_{\nu=0}^{m}a_\nu z^\nu\in\C[z],
 \qquad a_0\ne0,
\]
where $a_m\ne0$.

\begin{lemma}\label{lem:polynomial-degree}
For every coprime pair $1\le r<q$, the polynomial fibre determinant satisfies
\[
 \deg D^f_{r,q}=qm+\ell_q.
\]
Its leading coefficient has modulus
\begin{equation}\label{eq:polynomial-leading}
 \frac{|a_m|^q}{q}
 \prod_{\substack{\lambda^q=1\\\lambda\ne1}}
       |1-m+m\lambda|.
\end{equation}
For $m=0$ this is $|a_0|^q/q$; for $m\ge1$ it equals
\[
 \frac{|a_m|^q}{q}\bigl(m^q-(m-1)^q\bigr).
\]
\end{lemma}

\begin{proof}
Use the notation $c$, $P_r$, and $A(t)$ from Lemma
\ref{lem:exact-degree}, and put
\[
 v_f(t)=\bigl(f(t+c_\nu)\bigr)_{\nu=0}^{q-1}.
\]
The columns of the fibre matrix are
\[
 v_f,A(t)v_f,\ldots,A(t)^{q-1}v_f.
\]
Subtracting $t$ times column $j-1$ from column $j$, from right to left,
rewrites its determinant as
\begin{equation}\label{eq:polynomial-column-reduction}
 D^f_{r,q}(t)=
 \det[v_f,u_f,A(t)u_f,\ldots,A(t)^{q-2}u_f],
 \qquad u_f=(A(t)-tI)v_f.
\end{equation}
If $m=0$, then $v_f=a_0\mathbf1$ and
$D^f_{r,q}=a_0^qD_{r,q}$, so the result follows from Lemma
\ref{lem:exact-degree}.  Assume from now on that $m\ge1$.  Since
\[
 v_f(t)=a_mt^m\mathbf1+
 t^{m-1}(ma_mc+a_{m-1}\mathbf1)+O(t^{m-2}),
\]
one obtains
\[
 u_f(t)=a_mt^m d_m+O(t^{m-1}),
 \qquad d_m=mP_rc+(1-m)c.
\]
Also
\[
 A(t)^ku_f(t)=a_mt^{m+k}P_r^kd_m+O(t^{m+k-1}).
\]
Thus \eqref{eq:polynomial-column-reduction} has degree at most
$qm+\ell_q$, and the coefficient in that degree is
\begin{equation}\label{eq:polynomial-leading-det}
 a_m^q\det[\mathbf1,d_m,P_rd_m,\ldots,P_r^{q-2}d_m].
\end{equation}
On every nonconstant Fourier mode with eigenvalue $\lambda$ of $P_r$, passing
from $c$ to $d_m=(1-m+mP_r)c$ multiplies the Fourier coefficient by
$1-m+m\lambda$.  Since $P_r$ has all the $q$th roots of unity as its
eigenvalues, Lemma \ref{lem:exact-degree} shows that the modulus of
\eqref{eq:polynomial-leading-det} is exactly
\eqref{eq:polynomial-leading}.  None of the factors vanishes: this is immediate
for $m=1$, while for $m\ge2$ a zero would require a unimodular number to
equal $(m-1)/m$.  Hence the degree is exact.  Finally,
\[
 \prod_{\substack{\lambda^q=1\\\lambda\ne1}}|1-m+m\lambda|
 =m^{q-1}\frac{1-((m-1)/m)^q}{1-(m-1)/m}
 =m^q-(m-1)^q
\]
for $m\ge1$.
\end{proof}

Define the explicit bounds
\begin{equation}\label{eq:polynomial-bounds}
 B_0=\sum_{\nu=0}^{m}|a_\nu|,
 \qquad
 B_1=\sum_{\nu=1}^{m}\nu|a_\nu|,
 \qquad
 A_{f,p}=2^{1/p}B_0+
 \bigl(B_1^p+(2B_0)^p\bigr)^{1/p}.
\end{equation}
For $q\ge2$, retain $s_q,\varepsilon_q,W_q,a_q$ from Section~\ref{sec:gap-statement} and put
\begin{align*}
 \kappa_{f,q}&=qm+\ell_q,
 &b_{f,q}&=|a_0|^q q^{-s_q},\\
 \eta_{f,q}&=\frac{b_{f,q}}2
   \left(\frac{\varepsilon_q^p}{8}\right)^{\kappa_{f,q}},
 &J_{f,q}&=qB_0+B_1,\\
 U_{f,q}&=q!B_0^{q-1},
 &R_{f,q}&=q!qJ_{f,q}B_0^{q-1},\\
 V_{f,q}&=q!\left((q-1)J_{f,q}B_0^{q-2}
                   +\pi qB_0^{q-1}\right),
 &H_{f,q}&=\frac{U_{f,q}}{\eta_{f,q}}.
\end{align*}
Set
\begin{align}
 L_{f,q}&=\frac{H_{f,q}}{a_q}
          +\frac{V_{f,q}}{\eta_{f,q}}
          +\frac{2U_{f,q}R_{f,q}}{\eta_{f,q}^2},\label{eq:Lf}\\
 \Lambda_{f,q}&=\max\left\{1,
  \frac{\pi qB_0L_{f,q}}{\varepsilon_q}\right\},
 &\qquad N_{f,q}&=\left\lceil4\Lambda_{f,q}
  \log(4\Lambda_{f,q}+\mathrm e)\right\rceil,\label{eq:Nf0}\\
 d_{f,q}&=qN_{f,q}+q-1,\label{eq:Nf}\\
 C_{f,q}&=3qH_{f,q}N_{f,q}
  \Gamma_q^{N_{f,q}},\label{eq:Cf}\\
 \delta_{f,q}&=\min\left\{
  \frac1{2d_{f,q}},
  \frac{\varepsilon_q^p}
  {2A_{f,p}^pC_{f,q}^pd_{f,q}^{p+1}}
 \right\}.\label{eq:deltaf}
\end{align}
Finally define
\begin{equation}\label{eq:psif}
 \boxed{
  \psi_{f,p}(1)=1,
  \qquad
  \psi_{f,p}(q)=\frac1{q\delta_{f,q}}\quad(q\ge2).
 }
\end{equation}

\begin{theorem}\label{thm:polynomial}
Let $f\in\C[z]$ satisfy $f(0)\ne0$, let
$\alpha\in(0,1)\setminus\Q$, and let $(p_n/q_n)$ be the convergents of
$\alpha$.  If arbitrarily far out in the sequence one has
\[
 q_{n+1}>\psi_{f,p}(q_n),
\]
then $f$ is cyclic for $T_\alpha$ on $L^p[0,1]$.
\end{theorem}

\begin{proof}
Fix coprime integers $1\le r<q$ and form the fibre matrix
\[
 M^f_{r,q}(t)_{\nu j}
   =(T_{r/q}^{j}f)(t+\nu/q),
 \qquad 0\le\nu,j\le q-1.
\]
Lemma \ref{lem:polynomial-degree} gives
\[
 \deg D^f_{r,q}=\kappa_{f,q}=qm+\ell_q.
\]
B\'ehani's determinant calculation at the origin gives
\begin{equation}\label{eq:Df0}
 |D^f_{r,q}(0)|
 =|a_0|^q\prod_{k=1}^{q-1}
   \left\{\frac{kr}{q}\right\}^k
 \ge b_{f,q}.
\end{equation}
If $\kappa_{f,q}=0$, then $q=2$ and $f$ is constant, so the determinant is
constant and the sublevel set below is empty.  Otherwise choose a unimodular
number $\gamma_{r,q}$ such that
$\gamma_{r,q}D^f_{r,q}(0)>0$.  Applying the real Remez inequality to
$\operatorname{Re}(\gamma_{r,q}D^f_{r,q})$ shows that
\begin{equation}\label{eq:Ef}
 \left|\left\{t\in I_q:|D^f_{r,q}(t)|<2\eta_{f,q}\right\}\right|
 \le\frac{\varepsilon_q^p}{2q}.
\end{equation}
Indeed, the set on the left is contained in the corresponding sublevel set of
the rotated real part, while its value at zero has modulus at least $b_{f,q}$.

On $I_q$ one has
\begin{equation}\label{eq:Mf-bounds}
 |M^f_{r,q}(t)_{\nu j}|\le B_0,
 \qquad
 |(M^f_{r,q})'(t)_{\nu j}|\le J_{f,q}.
\end{equation}
With $G_n$ as in Section~\ref{sec:rational-approx} and
$u=(\operatorname{adj}M^f_{r,q})G_n$, the Leibniz formula therefore gives
\begin{equation}\label{eq:uf-bounds}
 |u_j|\le U_{f,q},
 \qquad |u_j'|\le V_{f,q},
 \qquad |(D^f_{r,q})'|\le R_{f,q}.
\end{equation}
Let $\rho_q$ be the endpoint cutoff from
Section~\ref{sec:rational-approx}, let $\chi_q$ now have
thresholds $\eta_{f,q}$ and $2\eta_{f,q}$, and define
\[
 h=\rho_q\frac{\chi_q(|D^f_{r,q}|)}{D^f_{r,q}}
   (\operatorname{adj}M^f_{r,q})G_n,
\]
with value zero where the determinant vanishes.  The scalar function
$z\mapsto\chi_q(|z|)/z$ has modulus at most $1/\eta_{f,q}$ and, along the
polynomial $D^f_{r,q}(t)$, Lipschitz constant at most
$2R_{f,q}/\eta_{f,q}^2$.  Thus
\begin{equation}\label{eq:hf-bounds}
 |h_j|\le H_{f,q},
 \qquad \operatorname{Lip}(h_j)\le L_{f,q}.
\end{equation}
The small-determinant set in \eqref{eq:Ef} and the endpoint cutoff together
occupy measure at most $\varepsilon_q^p$ after periodic repetition.  It follows
that
\begin{equation}\label{eq:f-cutoff-error}
 \left\|\sum_{j=0}^{q-1}h_jT_{r/q}^{j}f-g_n\right\|_p
 \le\varepsilon_q.
\end{equation}

As before, write each $h_j$ as a continuous function of $w_q$ on $I_q$, take
the Fej\'er polynomial of order $N_{f,q}$ after the cosine change of variables,
and apply the endpoint correction \eqref{eq:endpoint-correction}.  Equations
\eqref{eq:hf-bounds} and \eqref{eq:Nf0} give
\[
 \|B_j(w_q)-h_j\|_\infty
 \le\frac{\varepsilon_q}{qB_0}.
\]
Consequently the polynomial
\[
 Q^f_{r,q,n}(\xi)=\sum_{j=0}^{q-1}\xi^jB_j(\xi^q)
\]
satisfies
\begin{equation}\label{eq:f-rational-approx}
 \|Q^f_{r,q,n}(T_{r/q})f-g_n\|_p\le2\varepsilon_q,
\end{equation}
has degree at most $d_{f,q}$, and the Chebyshev coefficient calculation from
\eqref{eq:coeff} gives coefficient $\ell^1$-norm at most $C_{f,q}$.

It remains to make parameter continuity effective.  Put
\[
 P_{k,\beta}(x)=\prod_{j=0}^{k-1}\{x+j\beta\}.
\]
For $s=|\alpha-\beta|$ and $ks\le1$, the product estimate from
\eqref{eq:power-continuity} gives
\[
 \|P_{k,\alpha}-P_{k,\beta}\|_p
 \le2^{1/p}s^{1/p}k^{1+1/p}.
\]
The periodic extension of $f$ can jump at the endpoint.  Outside a set of
measure $ks$, its two arguments are at Euclidean distance $ks$; on the
exceptional set both values are bounded by $B_0$.  Hence
\begin{equation}\label{eq:f-translation}
 \|f(\{\cdot+k\alpha\})-f(\{\cdot+k\beta\})\|_p^p
 \le B_1^p(ks)^p+(2B_0)^pks.
\end{equation}
Combining the last two estimates yields, for $k\ge1$,
\begin{equation}\label{eq:f-power-continuity}
 \|T_\alpha^kf-T_\beta^kf\|_p
 \le A_{f,p}s^{1/p}k^{1+1/p}.
\end{equation}
Thus any polynomial $Q$ of degree at most $d_{f,q}$ and coefficient
$\ell^1$-norm at most $C_{f,q}$ satisfies
\[
 \|Q(T_\alpha)f-Q(T_\beta)f\|_p
 \le A_{f,p}s^{1/p}C_{f,q}d_{f,q}^{1+1/p}
 <\varepsilon_q
\]
whenever $s<\delta_{f,q}$, by \eqref{eq:deltaf}.

Finally fix $n$ and apply this estimate with $\beta=p_k/q_k$ along
arbitrarily large convergents for which $q_{k+1}>\psi_{f,p}(q_k)$.  For all
sufficiently large such indices, $n\le q_k$ and $0<p_k<q_k$.  Moreover,
\[
 \left|\alpha-\frac{p_k}{q_k}\right|
 <\frac1{q_kq_{k+1}}<\delta_{f,q_k}.
\]
Combining this with \eqref{eq:f-rational-approx} gives
\[
 \|Q^f_{p_k,q_k,n}(T_\alpha)f-g_n\|_p<3\varepsilon_{q_k},
\]
which tends to zero.  Thus each $g_n$ belongs to the closed cyclic subspace
generated by $f$, and $f$ is cyclic.
\end{proof}

Lemma \ref{lem:polynomial-degree} shows that the saving of $q-1$ degrees in
Lemma \ref{lem:exact-degree} persists for every polynomial vector.
For fixed $f$ and $p$, the polynomial error scale gives
\[
 -\log\eta_{f,q}=O(q^2\log q),
 \qquad
 \log\log\psi_{f,p}(q)=O(q^2\log q).
\]
The direct exponentially small error scale combined with Bernstein
approximation would instead give an inner exponent of order $q^3$.

\clearpage
\section{Scope and limitations}

The paper has two logically complementary components.  The Fourier
factorization, the exact example in Proposition~\ref{prop:q16}, and the
analytic infinite family in Theorem~\ref{thm:threeblock} show that universal nonvanishing and
monotonicity fail.  The Remez--Fej\'er argument shows that these zeros do not
prevent an explicit sufficient cyclicity criterion: only a quantitative
measure estimate for the small-determinant set is required.

Proposition~\ref{prop:q16} supplies a finite certificate consisting only of
explicit rational inequalities, whereas Theorem~\ref{thm:threeblock} gives an
asymptotic analytic mechanism.  Grivaux's complementary result
\cite{Grivaux} is incorporated in Section~\ref{sec:grivaux}, where its precise
intersection with the present determinant family is identified.

The proof avoids the unproved assertion that $D_{r,q}$ is zero-free and monotone
on $[0,1/q)$.  It uses only the nonzero value at the origin, the exact degree,
and Remez's inequality.  Thus isolated zeros cause no problem.

The resulting functions $\psi_{f,p}$ grow too rapidly to overlap with the
known upper-growth condition for hyperinvariant subspaces.  Thus the present
gap function answers the explicitness problem in Question~5.4 of
\cite{Behani}, but it is not strong enough to resolve Question~5.5: the method
does not yet exhibit an irrational parameter for which the Bishop operator is
both cyclic and has a nontrivial hyperinvariant subspace.

\end{document}